\documentclass{amsart}
\usepackage{amstext,amssymb,amsthm,amsopn,newlfont,graphpap,graphics,graphicx,mathrsfs,enumitem,tikz-cd}
\usepackage[parfill]{parskip}
\usepackage[noadjust]{cite}
\usepackage{epigraph}
\usepackage[colorlinks=true,
            linkcolor=red,
            urlcolor=blue,
            citecolor=magenta]{hyperref}
\usepackage{color}
\usepackage{mathrsfs}
\allowdisplaybreaks
\theoremstyle{plain}
\newtheorem{thm}{Theorem}[section]
\newtheorem*{thm*}{Theorem}
\newtheorem{prop}{Proposition}[section]
\newtheorem*{prop*}{Proposition}
\newtheorem{cor}{Corollary}[section]
\newtheorem*{cor*}{Corollary}

\newtheorem*{lem*}{Lemma}
\theoremstyle{definition}

\newtheorem*{defn*}{Definition}

\newtheorem*{exmp*}{Example}

\newtheorem*{exmps*}{Examples}

\newtheorem{rem}{Remark}[section]
\newtheorem*{rem*}{Remark}
\newtheorem{rems}{Remarks}[section]
\newtheorem*{rems*}{Remarks}

\newtheorem*{note*}{Note}
\newcommand{\N}{{\mathbb N}}
\newcommand{\Z}{{\mathbb Z}}
\newcommand{\R}{{\mathbb R}}
\newcommand{\C}{{\mathbb C}}
\newcommand{\F}{{\mathbb F}}
\newcommand{\emps}{\emptyset}

\newcommand\restr[2]{\ensuremath{#1\big|{#2}}}
\newcommand{\la}{\langle}
\newcommand{\ra}{\rangle}
\numberwithin{equation}{section}

\DeclareMathOperator{\spa}{span}
\DeclareMathOperator{\orb}{orb}
 
\DeclareMathOperator{\Per}{Per} 
\begin{document}
\title[On the chaoticity of derivatives]
{On the chaoticity of derivatives}
\author[Marat V. Markin]{Marat V. Markin}
\address{
Department of Mathematics\newline
California State University, Fresno\newline
5245 N. Backer Avenue, M/S PB 108\newline
Fresno, CA 93740-8001
}
\email{mmarkin@csufresno.edu}
\begin{abstract}
We utilize a recently established by the author \textit{sufficient condition for linear chaos} to prove the chaoticity of derivatives in the spaces $C[a,b]$ and $L_p(a,b)$ ($-\infty<a<b<\infty$, $1\le p<\infty$).
\end{abstract}
\maketitle
\epigraph{\textit{In all chaos there is a cosmos, in all disorder a secret order.}}{Carl Yung}

\section[Introduction]{Introduction}

The linear operator of differentiation 
\[
Df:=f'
\]
with maximal domain has been considered in various settings and shown to be \textit{chaotic}
\begin{itemize}
\item on the Fr\'echet space $H(\C)$ of entire functions with the topology of uniform convergence on compact subsets \cite{MacLane,Godefroy-Shapiro1991},
\item in the Hardy space $H^2$ \cite{B-Ch-S2001}, and
\item in the Bargmann space $F(\C)$ \cite{Emam-Hesh2005},
\end{itemize}
being continuous on $H(\C)$ while unbounded in $H^2$ and $F(\C)$.


In the complex space $L_2({\mathbb R})$, the unbounded differentiation operator $D$ with maximal domain has been shown to be \textit{non-hypercyclic} because of  being \textit{normal} \cite{Mark-Sich2019(1)} (cf. \cite{Markin2020(1)}).


It is to be noted that, in all the above cases, whenever $D$ is chaotic and, provided the underlying space is complex, each $\lambda \in \C$ is a simple eigenvalue for $D$.
 
We utilize a recently established by the author \textit{sufficient condition for linear chaos} {\cite[Theorem $3.2$]{arXiv:2106.14872}} to prove the \textit{chaoticity} of the differentiation operators
\[
D^nf:=f^{(n)}
\]
($n\in \N$) ($\N:=\left\{1,2,\dots\right\}$ is the set of \textit{natural numbers}) with maximal domains in the (real or complex) Banach spaces $C[a,b]$ and $L_p(a,b)$ ($-\infty<a<b<\infty$, $1\le p<\infty$), the former equipped with the maximum norm
\[
C[a,b]\ni f\mapsto \|f\|_\infty:=\max_{a\le x\le b}|f(x)|.
\]

It is noteworthy that the setting offered by the foregoing spaces does not allow reducing the problem of the chaoticity for the derivatives to that of weighted backward shifts (cf. \cite{Godefroy-Shapiro1991,Grosse-Erdmann2000,B-Ch-S2001,Emam-Hesh2005,B-B-T2008}). 


\section[Preliminaries]{Preliminaries}

The subsequent preliminaries are essential for our discourse.

\subsection{Hypercyclicity and Chaoticity}\

For a (bounded or unbounded) linear operator $A$ in a (real or complex) Banach space $X$, a nonzero vector 
\begin{equation*}
f\in C^\infty(A):=\bigcap_{n=0}^{\infty}D(A^n)
\end{equation*}
($D(\cdot)$ is the \textit{domain} of an operator, $A^0:=I$, $I$ is the \textit{identity operator} on $X$) is called \textit{hypercyclic} if its \textit{orbit} under $A$
\[
\orb(f,A):=\left\{A^nf\right\}_{n\in\Z_+}
\]
($\Z_+:=\left\{0,1,2,\dots\right\}$ is the set of \textit{nonnegative integers}) is dense in $X$.

Linear operators possessing hypercyclic vectors are said to be \textit{hypercyclic}.

If there exist an $N\in \N$ ($\N:=\left\{1,2,\dots\right\}$ is the set of \textit{natural numbers}) and a vector 
\[
f\in D(A^N)\quad \text{with}\quad A^Nf = f,
\]
such a vector is called a \textit{periodic point} for the operator $A$ of period $N$. If $f\ne 0$, we say that $N$ is a \textit{period} for $A$.

Hypercyclic linear operators with a dense in $X$ set $\Per(A)$ of periodic points are said to be \textit{chaotic}.

See \cite{Devaney,Godefroy-Shapiro1991,B-Ch-S2001}.

\begin{rems}\label{HCrems}\
\begin{itemize}
\item In the prior definition of hypercyclicity, the underlying space is necessarily
\textit{infinite-dimensional} and \textit{separable} (see, e.g., \cite{Grosse-Erdmann-Manguillot}).
\item For a hypercyclic linear operator $A$, the set $HC(A)$ of its hypercyclic vectors is necessarily dense in $X$, and hence, the more so, is the subspace $C^\infty(A)\supseteq HC(A)$.
\item Observe that
\[
\Per(A)=\bigcup_{N=1}^\infty \Per_N(A),
\]
where 
\[
\Per_N(A)=\ker(A^N-I),\ N\in \N
\]
is the \textit{subspace} of $N$-periodic points of $A$.
\item As immediately follows from the inclusions
\begin{equation*}
HC(A^n)\subseteq HC(A),\ \Per(A^n)\subseteq \Per(A), n\in \N,
\end{equation*}
if, for a linear operator $A$ in an infinite-dimensional separable Banach space $X$ and some $n\ge 2$, the operator $A^n$ is hypercyclic or chaotic, then $A$ is also hypercyclic or chaotic, respectively.
\end{itemize} 
\end{rems}

Prior to \cite{B-Ch-S2001,deL-E-G-E2003}, the notions of linear hypercyclicity and chaos had been studied exclusively for \textit{continuous} linear operators on Fr\'echet spaces, in particular for \textit{bounded} linear operators on Banach spaces (for a comprehensive survey, see \cite{Bayart-Matheron,Grosse-Erdmann-Manguillot}).

The following  statement, obtained in \cite{arXiv:2106.14872} by strengthening one of the hypotheses of a well-known sufficient condition for linear hypercyclicity \cite[Theorem $2.1$]{B-Ch-S2001} is a shortcut for establish chaoticity for (bounded or unbounded) linear operators without explicitly constructing both hypercyclic vectors and a dense set periodic points for them.

\begin{thm}[Sufficient Condition for Linear Chaos {\cite[Theorem $3.2$]{arXiv:2106.14872}}]\label{SCC}\ \\
Let $(X,\|\cdot\|)$ be a  (real or complex) infinite-dimensional separable Banach space and $A$ be a densely defined linear operator in $X$ such that each power $A^{n}$ ($n\in\N$) is a closed operator. If there exists a set
\[
Y\subseteq C^\infty(A):=\bigcap_{n=1}^\infty D(A^n)
\]
dense in $X$ and a mapping $B:Y\to Y$ such that
\begin{enumerate}
\item $\forall\, f\in Y:\ ABf=f$ and
\item $\forall\, f\in Y\  \exists\, \alpha=\alpha(f)\in (0,1),\ \exists\, c=c(f,\alpha)>0\ \forall\, n\in \N:$
\begin{equation*}
\max\left(\|A^nf\|,\|B^nf\|\right)\le c\alpha^n,
\end{equation*}
or equivalently,
\begin{equation}\label{(2(b))}
\forall\, f\in Y:\ \max\left(r(A,f),r(B,f)\right)<1,
\end{equation}
where 
\[
r(A,f):=\limsup_{n\to \infty}{\|A^nf\|}^{1/n}\quad \text{and}\quad
r(B,f):=\limsup_{n\to \infty}{\|B^nf\|}^{1/n},
\]
\end{enumerate}
then the operator $A$ is chaotic.
\end{thm}

We also need the subsequent

\begin{cor}[Chaoticity of Powers {\cite[Corollary $4.3$]{arXiv:2106.14872}}]\label{CP}\ \\
For a chaotic linear operator $A$ in a  (real or complex) infinite-dimensional separable Banach space $(X,\|\cdot\|)$ subject to the \textit{Sufficient Condition for Linear Chaos} (Theorem \ref{SCC}), each power $A^n$ ($n\in \N$) is chaotic.
\end{cor}

\subsection{Resolvent Set and Spectrum}\

For a linear operator $A$ in a complex Banach space $X$, the set
\[
\rho(A):=\left\{ \lambda\in \C \,\middle|\, \exists\, (A-\lambda I)^{-1}\in L(X) \right\}
\]
($L(X)$ is the space of bounded linear operators on $X$) and its complement $\sigma(A):=\C\setminus \rho(A)$ are called the operator's \textit{resolvent set} and \textit{spectrum}, respectively.

The spectrum of $A$ contains the set $\sigma_p(A)$ of all its eigenvalues, called its \textit{point spectrum} of $A$ (see, e.g., \cite{MarkinEOT,Dun-SchI}).

\begin{rem}\label{CPrem}
For an unbounded linear operator $A$ in a complex Banach space $X$ with a nonempty \textit{resolvent set} $\rho(A)\neq \emps$ (i.e., $\sigma(A)\neq \C$), all powers $A^{n}$ ($n\in\N$) are \textit{closed operators} \cite{Dun-SchI}.
\end{rem}

\subsection{$L_p$ Spaces}\

Henceforth, the notations $f$ and $f(\cdot)$ are used to designate an equivalence class in $L_p(a,b)$ ($1\le p<\infty$, $-\infty<a<b<\infty$) and its representative, respectively.

The inclusions
\begin{equation*}
C[a,b],C^\infty[a,b],C_0^\infty[a,b],P\subseteq L_p(a,b),
\end{equation*}
where
\begin{equation}\label{C0infty}
C_0^\infty[a,b]:=\left\{ f(\cdot)\in C^\infty[a,b]\,\middle|\,  f^{(n)}(a)=f^{(n)}(b)=0,\, n\in\Z_+\right\}
\end{equation}
and
\begin{equation}\label{P}
P:=\left\{ \sum_{k=0}^{n}c_kx^k\,\middle|\, n\in \Z_+,\ c_k\in \F, \, k=0,\dots,n,\ x\in [a,b]\right\}
\end{equation}
($\F:=\R$ or $\F:=\C$) is the subspace of polynomials, are understood in the sense of the natural embedding
\[
C[a,b]\in f\mapsto i(f):= f\in L_p(a,b),
\] 
where $i(f)$ is the equivalence class represented by the continuous function $f$.

\section{Chaoticity of Derivatives}

\subsection{Derivatives in $C[a,b]$}

\begin{prop}[Derivatives in  {$C[a,b]$}]\label{PDCab}\ \\
In the (real or complex) space $(C[a,b],\|\cdot\|_\infty)$ ($-\infty<a<b<\infty$),
the $n$th derivative
\[
D^nf:=f^{(n)}
\]
with maximal domain $D(D^n):=C^n[a,b]$ is a densely defined unbounded closed linear operator for every $n\in \N$.

Furthermore, each $\lambda \in \F$ is an eigenvalue for $D^n$ of geometric multiplicity $n$., i.e.,
\[
\dim\ker(D^n-\lambda I)=n.
\]
\end{prop}

\begin{proof}
Let $n\in \N$ be arbitrary.

Observe that $D^n$ is the $n$th power of $D$, and hence, is \textit{linear}.

The fact that $D^n$ is \textit{densely defined} follows from the inclusion
\begin{equation*}
P\subseteq C^\infty[a,b]= C^\infty(D),
\end{equation*}
where $P$ is the \textit{subspace of polynomials} (see \eqref{P}), which, by the \textit{Weierstrass approximation theorem}, is \textit{dense} in $(C[a,b],\|\cdot\|_\infty)$ (see, e.g., \cite{MarkinEFA}).

The \textit{unboundedness} of $D^n$ instantly follows from the fact that, for
\[
e_k(x):=\left(\frac{x-a}{b-a}\right)^k,\ k\in \N,x\in [a,b],
\]
we have:
\[
e_k\in D(D^n)\ \text{and}\ \|e_k\|_\infty=e_k(b)=1,\ k\in \N,
\]
and, for $k\ge n$,
\begin{align*}
\|D^{n}e_{k}\|_\infty&=\left\|\prod_{j=0}^{n-1}(k-j)\left(\frac{1}{b-a}\right)^{n}\left(\frac{x-a}{b-a}\right)^{k-n}\right\|_\infty\\
&=\frac{k!}{(k-n)!}\left(\frac{1}{b-a}\right)^{n}\left(\frac{x-a}{b-a}\right)^{k-n}\biggr|_{x=b}\\
&=\frac{k!}{(k-n)!}\left(\frac{1}{b-a}\right)^{n}
\ge k\left(\frac{1}{b-a}\right)^{n}\to \infty,\ k\to \infty.
\end{align*}

As is known from the theory of linear differential equations, for an arbitrary $\lambda \in \F$, the equation
\[
D^nf=\lambda f
\]
has $n$ lineally independent solutions $f_1,\dots,f_n\in C^\infty[a,b]=C^\infty(D)$. Hence, $\lambda$ is an eigenvalue for $A$, with the corresponding eigenspace
\[
\ker(D^n-\lambda I)=\spa\left(\left\{f_1,\dots,f_n\right\}\right)
\] 
being $n$-dimensional. 

Due to the above, the \textit{closedness} of the operator $D^n$ is not automatic (see Remark \ref{CPrem}), and hence, is to be shown.

Consistently with the \textit{Riesz representation theorem} (see, e.g., \cite{Dun-SchI,Goff-Ped,MarkinEFA}), there is a continuous embedding $E$ of $C[a,b]$ into the \textit{dual space} $C^*[a,b]$, which relates the corresponding vectors $g\in C[a,b]$ and $g^*:=Eg\in C^*[a,b]$ as follows:
\begin{equation}\label{IE}
\la f, g^*\ra=\int_a^b f(x)g(x)\,dx,\ f\in C[a,b],
\end{equation}
($\la\cdot,\cdot\ra$ is the pairing between $C[a,b]$ and $C^*[a,b]$) with
\[
\|g^*\|\le (b-a)\|g\|_\infty.
\]

In $C[a,b]$, the linear operator 
\begin{equation}\label{D_02}
D_0^nf:=f^{(n)}
\end{equation}
with domain
\begin{equation}\label{D(D_0)2}
D(D_0^n):=C_0^\infty[a,b]:=\left\{ f(\cdot)\in C^\infty[a,b]\,\middle|\,  f^{(n)}(a)=f^{(n)}(b)=0,\, n\in\Z_+\right\},
\end{equation}
is not densely defined, and hence, its adjoint (or conjugate) is not well defined. 

In $C^*[a,b]$, let us consider the linear operator ${(D_0^n)}'$ defined as follows:
\[
{(D_0^n)}'E:=ED_0^n, 
\]
i.e., via the commutative diagram
\begin{equation*}
\begin{tikzcd}[sep=large]
C^*[a,b]&\hspace{-1cm} \supseteq D({(D_0^n)}') \arrow[r, "{(D_0^n)}'"] & C^*[a,b]  \\
C[a,b] & \hspace{-1cm} \supseteq \quad D(D_0^n)\arrow[u, "E"] \arrow[u, "E"]\arrow[r, "D_0^n"]& C[a,b]\arrow[u, "E"]
\end{tikzcd},
\end{equation*}
for which
\[
D({(D_0^n)}'):=E(D(D_0^n)),
\]
and
\begin{equation}\label{**2}
\forall\, f\in C[a,b]\ \forall\, g\in D(D_0^n):\ \la f,{(D_0^n)}'g^*\ra=\int_a^b f(x)g^{(n)}(x)\,dx,
\end{equation}
where $g^*:=Eg\in D({(D_0^n)}')$.

The domain $D({(D_0^n)}')$ is a \textit{total} subspace of $C^*[a,b]$, i.e., a set separating 
points in $C[a,b]$, \cite[Definition II.$2.9$]{Goldberg}. Indeed, let $f\in C[a,b]$ and suppose that
\[
\forall\, g^*\in D({(D_0^n)}'):\ \la f,g^*\ra=\int_a^b f(x)g(x)\,dx=0,
\]
where $g:=E^{-1}g^*\in D(D_0^n)$, which implies that
\[
\forall\, g\in C_0^\infty[a,b]=D(D_0^n):\ \int_a^b f(x)g(x)\,dx=0,
\]
and hence, $f=0$ (see, e.g., \cite{Ziemer1989}).

Thus, for ${(D_0^n)}'$, well defined in $C[a,b]$ is the \textit{pre-adjoint} (or \textit{preconjugate}) operator
\begin{equation}\label{'*2}
\begin{split}
D({'{(D_0^n)}'}):=\{ f\in C[a,b]\,|\,&\exists\, h\in C[a,b]\ 
\forall\, g^*\in D({(D_0^n)}'):\\
& \la f,{(D_0^n)}'g^*\ra =\la h,g^*\ra \}\ni f\mapsto {'{(D_0^n)}'}f:=h.
\end{split}
\end{equation}
\cite[Definition VI.$1.1$]{Goldberg}.  

In view of \eqref{**2}, \eqref{'*2} acquires the form
\begin{equation}\label{''2}
\begin{split}
D({'{(D_0^n)}'})&:=\biggl\{ f\in C[a,b]\,\biggm|\,\exists\, h\in C[a,b]\ 
\forall\, g\in D(D_0^n):\\
& \int_a^b f(x)g^{(n)}(x)\,dx
=\int_a^b h(x)g(x)\,dx\biggr\}\ni f\mapsto {'{(D_0^n)}'}f:=h,
\end{split}
\end{equation}

From \eqref{''2}, we infer that $f\in D({'{(D_0^n)}'})\subseteq C[a,b]$ \textit{iff} its $n$th \textit{distributional derivative} $(-1)^nh$ belongs to $C[a,b]$, which is the case \textit{iff} 
\begin{equation*}
g(x)=(-1)^n\int_a^x\dots\int_a^{t_2}h(t_1)\,dt_1\dots\,dt_n
+\sum_{k=0}^{n-1}c_kx^k,\ x\in [a,b],
\end{equation*}
with some $c_k\in \F$, $k=0,\dots,n-1$, the latter being equivalent to the fact that 
\[
g\in C^n[a,b]\quad \text{and}\quad g^{(n)}(x)=(-1)^nh(x),\ x\in [a,b].
\]

Thus, we conclude that
\[
D({'{(D_0^n)}'})=C^n[a,b]=D(D^n)
\]
and
\[
{'{(D_0^n)}'}f=(-1)^nD^nf,\ f\in C^n[a,b],
\]
(see, e.g., \cite{Ziemer1989,Sobolev}), i.e.,
\[
D^n=(-1)^n{'{(D_0^n)}'}.
\]

By the closedness of a pre-adjoint operator \cite[Lemma VI.$1.2$]{Goldberg}, we infer that the operator $D^n$ is \textit{closed}, which completes the proof.
\end{proof}

\begin{thm}[Chaoticity of Derivatives in  {$C[a,b]$}]\label{CDCab}\ \\
In the (real or complex) space $(C[a,b],\|\cdot\|_\infty)$ ($-\infty<a<b<\infty$), 
the $n$th derivative
\[
D^nf:=f^{(n)}
\]
with maximal domain $D(D^n):=C^n[a,b]$ is a chaotic operator for every $n\in \N$.
\end{thm}

\begin{proof} 
By the prior proposition, $D^n$ is a densely defined unbounded closed linear operator for all
$n\in \N$.

Let
\begin{equation*}
Y:=P=\bigcup_{n=1}^\infty \ker D^n\subseteq C^\infty[a,b]= C^\infty(D),
\end{equation*}
where $P$ is the dense in $(C[a,b],\|\cdot\|_\infty)$ subspace of \textit{polynomials} (see \eqref{P}) and
\begin{equation*}
\ker D^n=\left\{ f\in P\,\middle|\,\deg f\le n-1  \right\},\ n\in \N.
\end{equation*}

The mapping $B:Y\to Y$ is the restriction to $Y$ of the Volterra integration operator
\[
[Bf](x):=\int_a^x f(t)\,dt,\ f\in C[a,b],x\in [a,b],
\]
which is a  \textit{quasinilpotent} bounded linear operator on $C[a,b]$, i.e.,
\begin{equation}\label{D}
\lim_{n\to \infty}{\|B^n\|}^{1/n}=0
\end{equation}
(here and henceforth, $\|\cdot\|$ also stands for the \textit{operator norm}) (see, e.g., \cite{MarkinEOT}).

Also,
\begin{equation*}
ABf=f,\ f\in C[a,b].
\end{equation*}

Let $f\in Y$ be arbitrary. For all $n\ge \deg f+1$,
\[
D^{n}f=0.
\]

Further, by \eqref{D}, we infer that
\begin{equation*}
\begin{aligned}
\forall\,  f\in C[a,b]:\ &0\le \limsup_{n\to \infty}{\|B^nf\|}^{1/n}
\le \limsup_{n\to \infty}{\left(\|B^n\|\|f\|\right)}^{1/n} \\
&=\lim_{n\to \infty}{\|B^n\|}^{1/n} \lim_{n\to \infty}{\|f\|}^{1/n}=0<1
\end{aligned}
\end{equation*}
(cf. \eqref{(2(b))}).

Thus, by the \textit{Sufficient Condition for Linear Chaos} (Theorem \ref{SCC}) and the \textit{Chaoticity of Powers} (Corollary \ref{CP}), for each $n\in \N$, the power $D^n$ is \textit{chaotic}.
\end{proof}

\begin{rems}\
\begin{itemize}
\item Since the natural embedding
\[
H(\C)\ni f\mapsto i(f):=\restr{f}{[a,b]}\in C^\infty[a,b]\subseteq C[a,b]
\]
($\restr{\cdot}{\cdot}$ is the \textit{restriction} of a function (left) to a set (right))
is \textit{continuous} and, due to the denseness of the subspace $P$ of polynomials in $(C[a,b],\|\cdot\|_\infty)$, has \textit{a dense range} (see, e.g., \cite{MarkinEFA}) and
\[
\forall\, f\in D(D_H)=H(\C): i(f)\in C^1[a,b]=D(D_C)\ \text{and}\ 
D_Ci(f)=i(D_Hf)
\]
($D_H$ and $D_C$ stand for the differentiation operators in $H(\C)$ and $C[a,b]$, respectively), i.e., the diagram
\begin{equation*}
\begin{tikzcd}[sep=large]
C[a,b] \hspace{-1cm} &\supseteq D(D_C) \arrow[r, "D_C"] & C[a,b]  \\
              & H(\C) \arrow[r, "D_H"]\arrow[u, "i"]& H(\C)\arrow[u, "i"]
\end{tikzcd}
\end{equation*}
commutes, the image $i(f)$ of any hypercyclic vector $f$ for $D_H$ is a hypercyclic vector for $D_C$. Also, the image $i(f)$ of an $N$-periodic point $f$ for $D_H$ is an $N$-periodic point for $D_C$.

Thus, the case of the differentiation operator $D$ in the complex space $C[a,b]$ follows from the chaoticity of  MacLane's operator $D$ in $H(\C)$ (see Introduction).
\item For the complex space $C[a,b]$ ($-\infty<a<b<\infty$), the fact that all $\lambda \in \C$ are eigenvalues for $D^n$ ($n\in \N$) of geometric multiplicity $n$ is consistent with \cite[Theorem $4.1$]{arXiv:2106.14872}. 
\end{itemize}
\end{rems}

\subsection{Derivatives in $L_p(a,b)$}\

The following statement has a value of its own and is not to be used to prove the chaoticity of derivatives in $L_p(a,b)$ ($1\le p<\infty$, $-\infty<a<b<\infty$).

\begin{prop}[Derivatives in  $L_p(a,b)$]\label{PDLp}\ \\
In the (real or complex) space $L_p(a,b)$ ($1\le p<\infty$, $-\infty<a<b<\infty$),
the $n$th derivative
\[
D^nf:=f^{(n)}
\]
with maximal domain
\begin{equation*}
\begin{split}
D(D^n):=W_p^n(a,b):=\bigl\{ f\in  L_p(a,b)\,\bigm|\,& f(\cdot)\in C^{n-1}[a,b],\ 
f^{(n-1)}(\cdot)\in AC[a,b],\\
&f^{(n)}\in L_p(a,b)\bigr\}
\end{split}
\end{equation*}
is a densely defined unbounded closed linear operator for every $n\in \N$.

Furthermore, each $\lambda \in \F$ is an eigenvalue for $D^n$ of geometric multiplicity $n$.
\end{prop}

\begin{proof}
Let $n\in \N$ be arbitrary.

Observe that $D^n$ is the $n$th power of $D$, and hence, is \textit{linear}.

The fact that $D^n$ is \textit{densely defined} follows from the inclusion
\begin{equation*}
P\subseteq C^\infty[a,b]= C^\infty(D),
\end{equation*}
where $P$ is the \textit{subspace of polynomials} (see \eqref{P}), which, by the \textit{Weierstrass approximation theorem}, is \textit{dense} in $(C[a,b],\|\cdot\|_\infty)$ (see, e.g., \cite{MarkinEFA}), and hence, in view of the denseness of $C[a,b]$ in $L_p(a,b)$ (see, e.g., \cite{MarkinRA}), also in $L_p(a,b)$.

The \textit{unboundedness} of $D^n$ follows from the fact that, for
\[
e_k(x):=\left(\frac{kp+1}{b-a}\right)^{1/p}\left(\frac{x-a}{b-a}\right)^k,\ k\in \N,x\in [a,b],
\]
we have:
\[
e_k\in D(D^n)\ \text{and}\ \|e_k\|_p=1,\ k\in \N,
\]
and, for $k\ge n$,
\begin{align*}
\|D^{n}e_{k}\|_p&=\left\|\left(\frac{kp+1}{b-a}\right)^{1/p}\prod_{j=0}^{n-1}(k-j)\left(\frac{1}{b-a}\right)^{n}\left(\frac{x-a}{b-a}\right)^{k-n}\right\|_p
\\
&=\left(\frac{kp+1}{b-a}\right)^{1/p}\frac{k!}{(k-n)!}\left(\frac{1}{b-a}\right)^{n}\left\|\left(\frac{x-a}{b-a}\right)^{k-n}\right\|_p\\
&=\left(\frac{kp+1}{b-a}\right)^{1/p}\frac{k!}{(k-n)!}\left(\frac{1}{b-a}\right)^{n}\left(\frac{b-a}{(k-n)p+1}\right)^{1/p}
\\
&\ge \left(\frac{kp+1}{(k-n)p+1}\right)^{1/p}k\left(\frac{1}{b-a}\right)^{n}\to \infty,\ k\to \infty.
\end{align*}

As we noted in the proof of the prior theorem, for an arbitrary $\lambda \in \F$, the equation
\[
D^nf=\lambda f
\]
has $n$ lineally independent solutions $f_1,\dots,f_n\in C^\infty[a,b]=C^\infty(D)$. Hence, $\lambda$ is an eigenvalue for $A$, with the corresponding eigenspace
\[
\ker(D^n-\lambda I)=\spa\left(\left\{f_1,\dots,f_n\right\}\right)
\] 
being $n$-dimensional. 

Due to the above, the \textit{closedness} of the operator $D^n$ is not automatic (see Remark \ref{CPrem}), and hence, is to be shown.

Let $q:=p/(p-1)\in (1,\infty]$ be the \textit{conjugate} to $p$ index.

As is known (see, e.g., \cite{Dun-SchI}), there is an \textit{isometric isomorphism} $E_{q,p}$ between the \textit{dual space} $L_q^*(a,b)$ and $L_p(a,b)$, which relates the corresponding vectors $g^*\in L_q^*(a,b)$ and $g:=E_{q,p}g^*\in L_p(a,b)$ as follows:
\begin{equation}\label{II}
\la f,g^*\ra=\int_a^b f(x)g(x)\,dx,\ f\in L_q(a,b)
\end{equation}
($\la\cdot,\cdot\ra$ is the pairing between $L_q(a,b)$ and $L_q^*(a,b)$).

In $L_q(a,b)$, consider the linear operator 
\begin{equation}\label{D_0}
D_0^nf:=f^{(n)}
\end{equation}
with domain
\begin{equation}\label{D(D_0)}
D(D_0^n):=C_0^\infty[a,b].
\end{equation}

Suppose first that $1<p<\infty$. Then $1<q<\infty$ and operator $D_0$ is \textit{densely defined} due to the denseness of $C_0^\infty[a,b]$ (see \eqref{C0infty}) in $L_q(a,b)$ (see, e.g., \cite{Ziemer1989}). Therefore, for $D_0^n$, well defined in $L_q^*(a,b)$ is the \textit{adjoint} (or \textit{conjugate}) operator
\begin{equation*}
\begin{split}[sep=large]
D({(D_0^n)}^*):=\{ g^*\in L_q^*(a,b)\,|\,& \exists\, h^*\in L_q^*(a,b):\\
&g^*D_0^n=\restr{h^*}{D(D_0^n)}\}
\ni g^*\mapsto {(D_0^n)}^*g^*:=h^*
\end{split}
\end{equation*}
\cite{Dun-SchI,Goldberg}.

Due to the isometric isomorphism $E_{q,p}$ between the \textit{dual space} $L_q^*(a,b)$ and $L_p(a,b)$ (see \eqref{II}), ${(D_0^n)}^*$ can be identified with a linear operator
${(D_0^n)}'$ in $L_p(a,b)$ defined as follows:
\[
{(D_0^n)}'E_{q,p}:=E_{q,p}{(D_0^n)}^*,
\]
i.e., via the commutative diagram
\begin{equation*}
\begin{tikzcd}[sep=large]
L_q^*(a,b)&\hspace{-1cm} \supseteq D({(D_0^n)}^*) \arrow[r, "{(D_0^n)}^*"] \arrow[d, "E_{q,p}"]& L_q^*(a,b)  \arrow[d, "E_{q,p}"]\\
L_p(a,b)&\hspace{-1cm} \supseteq D({(D_0^n)}') \arrow[r, "{(D_0^n)}'"]& L_p(a,b)
\end{tikzcd},
\end{equation*}
for which
\begin{equation}\label{'}
\begin{split}
D({(D_0^n)}')&:=E_{q,p}({(D_0^n)}^*)=\biggl\{ g\in L_p(a,b)\,\biggm|\, \exists\, h\in L_p(a,b)\ \forall\,f\in D(D_0^n):\ 
\\
&\int_a^b f^{(n)}(x)g(x)\,dx=\int_a^b f(x)h(x)\,dx\biggr\}\ni f\mapsto {(D_0^n)}'g:=h.
\end{split}
\end{equation}

By the closedness of an adjoint operator (see, e.g., \cite{Goldberg}), we conclude that the operator ${(D_0^n)}'$ is \textit{closed} along with ${(D_0^n)}^*$.

From \eqref{'}, we infer that $g\in D({(D_0^n)}')\subseteq L_p(a,b)$ \textit{iff} its $n$th \textit{distributional derivative} $(-1)^nh$ belongs to $L_p(a,b)$, which is the case \textit{iff} the equivalence class $g$ admits the representative
\begin{equation}\label{nDD}
g(x)=(-1)^n\int_a^x\dots\int_a^{t_2}h(t_1)\,dt_1\dots\,dt_n
+\sum_{k=0}^{n-1}c_kx^k,\ x\in [a,b],
\end{equation}
with some $c_k\in \F$, $k=0,\dots,n-1$, the latter being equivalent to the fact that 
\[
g\in W_p^n(a,b)\quad \text{and}\quad g^{(n)}(x)=(-1)^nh(x)\ \text{on}\ [a,b]  \pmod{\mu}
\]
($\mu$ is the Lebesgue measure on $\R$).

Thus, we conclude that
\[
D({(D_0^n)}')=W_p^n(a,b)=D(D^n)
\]
and
\[
{(D_0^n)}'f=(-1)^nD^nf,\ f\in W_p^n(a,b),
\]
(see, e.g., \cite{Ziemer1989,Sobolev}), which implies that
\[
D^n=(-1)^n{(D_0^n)}'.
\]

Whence, by the closedness of the operator ${(D_0^n)}'$, we infer that the operator $D^n$ is \textit{closed}.

For $p=1$, $q=\infty$ and the linear operator $D_0^n$ in $L_\infty(a,b)$ (see \eqref{D_0}--\eqref{D(D_0)} with $q=\infty$) is not densely defined, and hence, has no adjoint. 

In $L_1^*(a,b)$, let us consider a linear operator ${(D_0^n)}'$ defined as follows:
\[
{(D_0^n)}'E_{1,\infty}^{-1}:=E_{1,\infty}^{-1}D_0^n,
\]
where $E_{1,\infty}$ the isometric isomorphism between the \textit{dual space} $L_1^*(a,b)$ and $L_\infty(a,b)$ (see \eqref{II} with $q=1$ and $p=\infty$), 
i.e., via the commutative diagram
\begin{equation*}
\begin{tikzcd}[sep=large]
L_1^*(a,b)&\hspace{-1.1cm} \supseteq D({(D_0^n)}') \arrow[r, "{(D_0^n)}'"] & L_1^*(a,b) \\
L_\infty(a,b)&\hspace{-1cm} \supseteq D(D_0^n) \arrow[r, "D_0^n"]\arrow[u, "E_{1,\infty}^{-1}"]& L_\infty(a,b)
\arrow[u, "E_{1,\infty}^{-1}"]
\end{tikzcd},
\end{equation*}
for which
\[
D({(D_0^n)}'):=E_{1,\infty}^{-1}(D(D_0^n)),
\]
and
\begin{equation}\label{**}
\forall\, f\in L_1(a,b)\ \forall\, g\in D(D_0^n):\ ({(D_0^n)}'g^*)f=\int_a^b f(x)g^{(n)}(x)\,dx,
\end{equation}
where $g^*:=E_{1,\infty}^{-1}g\in D({(D_0^n)}')$.

The domain $D({(D_0^n)}')$ is a \textit{total} subspace of $L_1^*(a,b)$, i.e., a set separating points in $L_1(a,b)$, \cite[Definition II.$2.9$]{Goldberg}. Indeed, let $f\in L_1(a,b)$ and suppose that
\[
\forall\, g^*\in D({(D_0^n)}'):\ \la f,g^*\ra =\int_a^b f(x)g(x)\,dx=0
\]
($\la\cdot,\cdot\ra$ is the pairing between $L_1(a,b)$ and $L_1^*(a,b)$), where $g:=E_{1,\infty}g^*\in D(D_0^n)$, which implies that
\[
\forall\, g\in C_0^\infty[a,b]=D(D_0^n):\ \int_a^b f(x)g(x)\,dx=0,
\]
and hence, $f=0$ (see, e.g., \cite{Ziemer1989}).

Thus, for ${(D_0^n)}'$, well defined in $L_1(a,b)$ is the \textit{pre-adjoint} (or \textit{preconjugate}) operator
\begin{equation}\label{'*}
\begin{split}
D({'{(D_0^n)}'}):=\{ f\in L_1(a,b)\,|\,&\exists\, h\in L_1(a,b)\ 
\forall\, g^*\in D({(D_0^n)}'):\\
& \left({(D_0^n)}'g^*\right)f=g^*h\}\ni f\mapsto {'{(D_0^n)}'}f:=h,
\end{split}
\end{equation}
\cite[Definition VI.$1.1$]{Goldberg}.

In view of \eqref{**}, \eqref{'*} acquires the form
\begin{equation}\label{''}
\begin{split}
D({'{(D_0^n)}'})&:=\biggl\{ f\in L_1(a,b)\,\biggm|\,\exists\, h\in L_1(a,b)\ 
\forall\, g\in D(D_0^n):\\
& \int_a^b f(x)g^{(n)}(x)\,dx
=\int_a^b h(x)g(x)\,dx\biggr\}\ni f\mapsto {'{(D_0^n)}'}f:=h,
\end{split}
\end{equation}

From \eqref{''}, we infer that $f\in D({'{(D_0^n)}'})\subseteq L_1(a,b)$ \textit{iff} its $n$th \textit{distributional derivative} $(-1)^nh$ belongs to $L_1(a,b)$. Whence, by reasoning as in regard to \eqref{nDD}, we can conclude that
\[
D({'{(D_0^n)}'})=W_1^n(a,b)=D(D^n)
\]
and
\[
{'{(D_0^n)}'}f=(-1)^nD^nf,\ f\in W_1^n(a,b),
\]
which implies that
\[
D^n=(-1)^n{'{(D_0^n)}'}.
\]

Whence, by the closedness of a pre-adjoint operator \cite[Lemma VI.$1.2$]{Goldberg}, we infer that the operator $D^n$ is \textit{closed}, which completes the proof.
\end{proof}

Since the Volterra integration operator 
\[
[Bf](x):=\int_a^x f(t)\,dt,\ f\in L_p(a,b),x\in [a,b],
\]
is also a  \textit{quasinilpotent} bounded linear operator on $L_p(a,b)$ ($1\le p<\infty$, $-\infty<a<b<\infty$) (see, e.g., \cite{Grosse-Erdmann-Manguillot}), the following statement can be proved by replicating the proof of Theorem \ref{CDCab} verbatim. 

\begin{thm}[Chaoticity of Derivatives in  $L_p(a,b)$]\label{CDLp}\ \\
In the (real or complex) space $L_p(a,b)$ ($1\le p<\infty$, $-\infty<a<b<\infty$), 
the $n$th derivative
\[
D^nf:=f^{(n)}
\]
with maximal domain $D(D^n):=W_p^n(a,b)$ is a chaotic operator for every $n\in \N$.
\end{thm}

\begin{samepage}
\begin{rems}\
\begin{itemize}
\item Since, for $1\le p<\infty$, the natural embedding
\[
(C[a,b],\|\cdot\|_\infty)\ni f\mapsto i(f):=f\in L_p(a,b)
\]
is \textit{continuous} and has \textit{a dense range} (see, e.g., \cite{Dun-SchI,MarkinRA}) and, for $n\in \N$,
\[
\forall\, f\in D(D_C^n)=C^n[a,b]: i(f)\in W_p^n(a,b)=D(D_{L_p}^n)\ \text{and}\ 
D_{L_p}^ni(f)=i(D_C^nf)
\]
($D_C^n$ and $D_{L_p}^n$ stand for the $n$th derivative operators with maximal domain in $C[a,b]$ and $L_p(a,b)$, respectively), i.e., the diagram
\begin{equation*}
\begin{tikzcd}[sep=large]
L_p(a,b)&\hspace{-1cm} \supseteq D(D_{L_p}^n) \arrow[r, "D_{L_p}^n"] & L_p(a,b)  \\
C[a,b]&\hspace{-1cm} \supseteq D(D_C^n) \arrow[r, "D_C^n"]\arrow[u, "i"]& C[a,b]\arrow[u, "i"]
\end{tikzcd}
\end{equation*}
commutes, the image $i(f)$ of any hypercyclic vector $f$ for $D_C^n$ is a hypercyclic vector for $D_{L_p}^n$. Also, the image $i(f)$ of an $N$-periodic point $f$ for $D_C^n$ is an $N$-periodic point for $D_{L_p}^n$.

Thus, the prior theorem can also be considered as a corollary of Theorem \ref{CDCab}. 
\item For the complex space $L_p[a,b]$ ($1\le p<\infty$, $-\infty<a<b<\infty$), the fact that all $\lambda \in \C$ are eigenvalues for $D^n$ ($n\in \N$) of geometric multiplicity $n$ is consistent with \cite[Theorem $4.1$]{arXiv:2106.14872}. 
\end{itemize}
\end{rems}
\end{samepage}

\section{Concluding Remarks}

\begin{itemize}
\item Since $C^n[a,b]$ ($n\in \N$) is a Banach space relative to the norm
\[
\|f\|_n:=\sum_{k=0}^n \|f^{(k)}\|_\infty,\ f\in C^n[a,b],
\]
(see, e.g., \cite{MarkinEOT}), the closedness of the $n$th derivative operator $D^n$ with $D(D^n)=C^n[a,b]$ in $C[a,b]$ implies that  $C^n[a,b]$ is also a Banach space relative to the weaker graph norm
\[
\|f\|_{D^n}:=\|f\|_\infty+\|f^{(n)}\|_\infty\le \|f\|_n,\ f\in C^n[a,b],
\]
and hence, as follows from the \textit{Inverse mapping theorem}, the norms $\|\cdot\|_n$ and $\|\cdot\|_{D^n}$ on $C^n[a,b]$ are \textit{equivalent}, i.e.,
\[
\exists\, C>0\ \forall\,f\in C^n[a,b]:\  \|f\|_n\le C\left[\|f\|_\infty+\|f^{(n)}\|_\infty\right]
\]
(see, e.g., \cite{MarkinEOT}).
\item Since the Sobolev space $W_p^n(a,b)$ ($1\le p<\infty$, $n\in \N$, $-\infty<a<b<\infty$) is a Banach space relative to the norm
\[
\|f\|_{p,n}:=\sum_{k=0}^n \|f^{(k)}\|_p,\ f\in W_p^n(a,b),
\]
(see, e.g., \cite{Sobolev,Ziemer1989}), the closedness of the $n$th derivative operator $D^n$ with $D(D^n)=W_p^n(a,b)$ in $L_p(a,b)$ similarly implies that the norm $\|\cdot\|_{p,n}$ on $W_p^n(a,b)$ is \textit{equivalent} to the weaker graph norm
\[
\|f\|_{D^n}:=\|f\|_p+\|f^{(n)}\|_p\le \|f\|_{p,n},\ f\in W_p^n(a,b),
\]
i.e.,
\[
\exists\, C>0\ \forall\,f\in W_p^n(a,b):\  \|f\|_{p,n}\le C\left[\|f\|_p+\|f^{(n)}\|_p\right].
\]
\end{itemize}

 
\end{document}